\documentclass[12pt,draft]{article} 
 \usepackage{amssymb}
\usepackage{amsthm}

\newtheorem{theorem}{Theorem}[section]
\newtheorem{lemma}[theorem]{Lemma}
\newtheorem{corollary}[theorem]{Corollary}

\title{On $d$-panconnected tournaments with large semidegrees}

\begin{document}
\date{}
 \maketitle

\vspace{-2cm}

\begin{center}
\noindent\textbf{Samvel Kh. Darbinyan}\\

Institute for Informatics and Automation Problems of NAS RA

\noindent\textbf {Gregory Z. Gutin}\\

Department of Compute Science, Royal Holloway, University of London, UK

E-mails: samdarbin@iiap.sci.am, gutin@cs.rhul.ac.uk\\

\end{center}

\begin{abstract} We prove the following new results. 

(a) Let $T$ be a regular tournament of order $2n+1\geq 11$ and  $S$ a subset of $V(T)$. Suppose that $|S|\leq \frac{1}{2}(n-2)$ and $x$, $y$ are distinct vertices in $V(T)\setminus S$. If the subtournament $T-S$ contains an $(x,y)$-path of length $r$, where $3\leq r\leq |V(T)\setminus S|-2$, then  $T-S$ also contains an $(x,y)$-path of length $r+1$.

(b) Let $T$ be an $m$-irregular tournament of order $p$, i.e., $|d^+(x)-d^-(x)|\le m$ for every vertex $x$ of $T.$  If 
$m\leq \frac{1}{3}(p-5)$ (respectively,  $m\leq \frac{1}{5}(p-3)$), then for every pair of vertices $x$ and $y$,  $T$ has an $(x,y)$-path of any length $k$, $4\leq k\leq p-1$ (respectively,  $3\leq k\leq p-1$ or $T$ belongs to a family $\cal G$ of tournaments, which is defined in the paper).
In other words, (b) means that if the semidegrees of every vertex of a tournament $T$  of order $p$ are between
 $\frac{1}{3}(p+1)$ and $\frac{2}{3}(p-2)$ (respectively, between $\frac{1}{5}(2p-1)$ and  $\frac{1}{5}(3p-4)$), then the claims in (b) hold.
 
 Our results  improve in  a sense related results of Alspach (1967), Jacobsen (1972), Alspach et al. (1974), Thomassen (1978) and Darbinyan (1977, 1978, 1979), and are sharp in a sense.
 
\end{abstract}

\section {Introduction} 
In this paper, we consider finite digraphs (directed graphs)  without loops and multiple arcs. We use standard notation and terminology, cf. \cite{[3]} and \cite{[4]}. The vertex set and the arc set of a digraph $D$ are    denoted
  by $V(D)$  and   $A(D)$, respectively.  The {\it order} of $D$ is the number
  of its vertices. A subdigraph of $D$ induced by a subset $A\subseteq  V(D)$ is denoted by $D\langle A\rangle$. If $X\subseteq V(D)$, then $D-X$ is the subdigraph induced by $V(D)\setminus X$, i.e., $D-X=D\langle V(D)\setminus X\rangle$. 
 Every cycle and path are assumed to be simple and directed.
 Let $m$ and $n$, $m\leq n$, are integers. By $[m,n]$ we denote the set $\{m, m+1,\ldots , n\}$.  
 
A digraph $D$ of order $p$ is {\it arc pancyclic} (respectively, {\it $d$-arc pancyclic}, where $d\in [3,p]$) if $D$ has a $k$-cycle containing $uv$ for every arc $uv\in A(D)$ and  every $k\in [3,p]$ (respectively, $k\in [d,p]$).
 We say that a digraph $D$ of order $p$ is {\it strongly panconnected} (respectively, {\it $d$-strongly panconnected}, where $d\in [3,p-1]$) if there is an $(x,y)$- and a $(y,x)$-path in $D$, both of length $k$, for any two vertices   $x$,  $y$  of $D$ and each $k\in [3,p-1]$ (respectively, $k\in [d,p-1]$).

An {\em oriented graph} is a digraph with no cycle of length two. A {\em tournament} is an oriented graph where every pair of distinct vertices are adjacent. The {\em outdegree} $d^+(x)$ (respectively, {\em indegree}  $d^-(x)$) of a vertex $x$ of a digraph $D$ is the number of vertices $y$ such that $xy\in A(D)$ (respectively, $yx\in A(D)$).
The {\it irregularity} $i(T)$ of a tournament $T$ is the maximum of $|d^+(x)-d^-(x)|$ over all vertices  $x$ of $T$. If $i(T)=0$, then $T$ is {\it regular} and if  $i(T)=1$, then $T$ is {\em almost regular}. If $i(T)= m$, $T$ is {\em $m$-irregular}. Observe that every vertex of a tournament $T$ of order $p$ has  outdegree between $\frac{1}{2}(p-1-i(T))$ and $\frac{1}{2}(p-1+i(T))$. The outdegree and indegree of $x$ are its {\em semidegrees}.
Further digraph terminology and notation are given in the next section. 

 There are a number of  conditions which guarantee that  a tournament is arc pancyclic or strongly panconnected (see, e.g., \cite{[1]}-\cite{[18]}).  In particular, Alspach \cite{[1]}  proved that every  regular tournament is arc pancyclic. Jacobsen \cite{[12]}  proved that every almost regular tournaments of order $p\geq 8$ is 4-arc pancyclic. Alspach et al. \cite{[2]}  proved that every regular tournament of order $p\geq 7$ is strongly panconnected. Darbinyan \cite{[9]} proved that every almost regular tournament of order $p\geq 10$ is strongly panconnected.

  Thomassen \cite{[17]} generalized these results
as follows:

\begin{theorem}\label{th1} Let $T$ be an $m$-irregular tournament of order $p$.  If 
 $m\leq \frac{1}{5}(p-9)$, then $T$ is strongly panconnected. If $m\leq \frac{1}{5}(p-3)$, then $T$ is 4-arc pancyclic.
 \end{theorem}
 
  In \cite{[8]}  and \cite{[10]}, Darbinyan obtained the following:
  
   \begin{theorem}\label{th2}   Let $T$ be a regular tournament of order $2n+1$ and let $S\subset V(T)$. 
   
   (i) \cite{[8]} If $2\leq |S|\leq \frac{1}{3}(n-2)$, then $T- S$ is strongly panconnected.
   
  (ii) \cite{[10]}  Let  $|S|\leq  \frac{1}{2}(n-3)$ and $x,y \in V(T)\setminus S$ be two distinct vertices. If 
   $T- S$ contains an $(x,y)$-path of length $r$, where $r\in [3, 2n-|S|-1]$, then  $T- S$ also contains an $(x,y)$-path of length $r+1$.
   \end{theorem}
   
We will use the following result of Moon \cite{[14]}. 
   
 \begin{theorem}\label{th3} Let $H$ be an $m$-irregular tournament of order $p\geq 2$. Then there is    a regular tournament of order $p+m$ such that $T$ contains $H$ as a subtournament. \end{theorem}

  From Theorems \ref{th2}(ii) and \ref{th3} it is not difficult to obtain the following:

\begin{corollary}\label{co1} \cite{[10]} Let $T$ be an $m$-irregular  tournament of order $p\geq 7$, where $m\leq \frac{1}{3} (p-7)$, and  $x$, $y$ are two distinct vertices.  If $T$ contains an  $(x,y)$-path of length $r$, where $r\in [3, p-2]$, then  $T$ also contains an $(x,y)$-path of length $r+1$. \end{corollary}

In this paper, we prove the following theorems, which improve in a sense  the above-mentioned results of Alspach, Jacobsen, Alspach et al., Thomassen and Darbinyan. The following theorem is our main result.
 
   \begin{theorem}\label{th4} Let $T$ be a regular tournament of order $2n+1\geq 11$ and let $S$ be a subset  in $V(T)$. Suppose that  $|S|\leq \frac{1}{2}(n-2)$ and $x,y$  be two distinct vertices in $V(T)\setminus S$. If 
   $T-S$ contains an $(x,y)$-path of length $r$, where $r\in [3, 2n-|S|-1]$, then  $T- S$ also contains an $(x,y)$-path of length $r+1$. \end{theorem}
   
%Theorem 1.5 implies all the main results of  \cite{[8]},  \cite{[9]} and  \cite{[10]}.  
Note that Theorem 1.2(i) was only announced in  \cite{[8]} and its proof has never been published. The preprint \cite{[10]} gave only an outline of the proof of Theorem 1.2(ii) and its complete proof has never been published either.  Also note that the main arguments in the proof of Theorem 1.5 are different from those in the proof of Theorem 1.1 by Thomassen.\\

\noindent\textbf{Remark 1} Let $H$ be a tournament of order $p\geq 5$  with irregularity $m\leq \frac{1}{3}(p-5)$. Then by  Theorem \ref{th3}, there is a regular tournament $T$ of order $p+m=2n+1$ such that $T$ contains $H$ as a subtournament. Let $S=V(T)\setminus V(H)$. Note that $|S|=m$. Then we have $p+ \frac{1}{3}(p-5)\geq 2n+1$ implying $p\geq \frac{1}{2}(3n+4)$. Therefore, $m=2n+1-p\leq 2n+1-\frac{1}{2}(3n+4)=\frac{1}{2}(n-2)$. Thus, $|S|=m\leq \frac{1}{2}(n-2)$.\\

%(ii). Let $T$ be a regular tournament of order $2n+1\geq 5$ and $S$ be a subset in $V(T)$ such that $|S|\leq \frac{1}{2}(n-2)$. Let $x$ be an arbitrary vertex in $V(T)\setminus S$. Assume that $d^+(x,S)=a_x$, where
% $0\leq a_x\leq |S|$. Put $H=T-S$. Then  $d^-(x,S)=|S|-a_x$,  $d^+(x,H)=n-a_x$ and  $d^-(x,H)=n-|S|+a_x$. Since $0\leq a_x\leq |S|$, observe that 
% $i(H)= \max_{x\in V(H)}|d^+(x,H)-d^-(x,H)|=\max_{x\in V(H)}||S|-2a_x|\leq |S|$. Thus $i(H)\leq |S|$. We have that $|S|+p=2n+1$, where $p$ is the order of $H$. Since $i(H)\leq |S|$, it follows that $i(H)+p\leq 2n+1$ and $i(H)\leq 2n+1-p$. On the other hand, $p\geq 2n+1- \frac{1}{2}(n-2)$ implying $n\leq \frac{1}{3}(2p-4)$. Therefore, $i(H)\leq |S|=2n+1-p\leq \frac{2}{3}(2p-4)+1-p=\frac{1}{3}(p-5)$. Thus we have that if $|S|\leq \frac{1}{2}(n-2)$, then $i(H)\leq \frac{1}{3}(p-5)$.\\

%Using  Remark 1, Theorem \ref{th4} and Lemma \ref{le2} in Section 3, we can obtain the following:

\begin{theorem}\label{th5} Let $T$ be an $m$-irregular tournament of order $p$ such that $p+m\ge 11$.  If 
$m\leq \frac{1}{3}(p-5)$ (respectively, 
 $m\leq \frac{1}{5}(p-3)$), then $T$ is 4-strongly panconnected (respectively,  $T$ is strongly panconnected or belongs to the family  $\cal G$ of tournaments defined in Section 3). \end{theorem}
 \begin{proof}
Use the construction of Remark 1 to obtain  a regular tournament $\hat{T}$ containing $T$. Let $2n+1=p+m$ be the order of $\hat{T}$. Let $S=V(\hat{T})\setminus V(T).$

Let us first prove the case of  $m\leq \frac{1}{3}(p-5)$.  By Remark 1,  $m=|S|\leq \frac{1}{2}(n-2)$ and hence by Lemma \ref{le2}(ii), $T$ has an $(x,y)$-path of length 3 or 4 for every pair $x,y$ of distinct vertices. 
Also, by Theorem \ref{th4}, $T$ has an $(x,y)$-path of any length from 4 to $|V(T)|-1$. Thus, $T$ is 4-strongly panconnected.

Now we prove the case of $m\leq \frac{1}{5}(p-3)$. Since $p+m\ge 11$ and $m\leq \frac{1}{5}(p-3)$, we have $p\ge 10.$ Hence, $m\leq \frac{1}{5}(p-3)\le \frac{1}{3}(p-5)$ and we can use Lemma \ref{le2} to show that
$T$ either has an $(x,y)$-path of length 3  for every pair $x,y$ of distinct vertices, or $T$ belongs to the family  $\cal G$. If $T$ does not belongs to $\cal G$, by Theorem \ref{th4}, $T$ is strongly panconnected.
\end{proof}

\noindent\textbf{Remark 2.} Using the above observation, we can formulate Theorem \ref{th5} in terms of semidegrees
 as follows:
{\it If the semidegrees of every vertex of a tournament $T$  of order $p$ are between
 $\frac{1}{3}(p+1)$ and $\frac{2}{3}(p-2)$ (respectively, between $\frac{1}{5}(2p-1)$ and  $\frac{1}{5}(3p-4)$), then the assertion of Theorem \ref{th5} holds}.\\
 
 From Theorem \ref{th5} (Theorem \ref{th4}) it follows that
 
 (i) if  $T$ is regular (i.e, $m=0$), then for $p\geq 11$ we have a result by Alspach that any regular tournament is arc pancyclic (observe that any arc of a regular tournament is on 3-cycle) and for $p\geq 11$ a result by Alspach et al.  that any regular tournament is strongly panconnected.
   
  (ii) if $T$ is almost regular (i.e, $m=1$), then for $p\geq 10$ we have a result by Jacobsen that any almost regular tournament is 4-arc pancyclic  and  a result by Darbinyan  that any almost regular tournament of order $p\geq 10$ is strongly panconnected.  
  
  (iii) from Theorem \ref{th5} (Theorem \ref{th4}) also we have Theorems \ref{th1} and \ref{th2}.\\
   
 The following remark shows that Theorem 1.5 is sharp in a sense.\\

\noindent\textbf{Remark 3.} If in Theorem \ref{th4}  we replace 
$\frac{1}{2}(n-2)$  with $\frac{1}{2}(n-1)$ (respectively, $\frac{n}{2}$), then there is a regular tournament of order $2n+1=11$ (respectively, $2n+1=13$), which contains a subset $S\subset V(T)$ with $|S|=\frac{1}{2}(n-1)$ (respectively, $|S|=\frac{n}{2}$) and two distinct vertices $x,y\in V(T)\setminus S$ such that $T-S$ contains an $(x,y)$-path of length 3, but $T-S$ contains no 
 $(x,y)$-path of length 4.
 
 Here we define only a regular tournament $T$ of order $2n+1=11$ as follows:
 
$V(T)=\{x_0,x_1,x_2, x_3\}\cup A\cup S$, where $A=\{u_1,u_2,u_3,u_4,z\}$ and
$S=\{v_1,v_2\}$. The arc set of $T$ is defined by in-neighborhoods of vertices as follows: $N^-(x_0)=A$, $N^-(x_1)=\{x_0,z,u_3,u_4,v_2\}$,
$N^-(x_2)=\{x_0,x_1,z,u_3,u_4\}$, $N^-(x_3)=\{x_0,x_1,x_2\}\cup S$, $N^-(z)=\{x_3,u_1,u_2,u_3,u_4\}$, 
$N^-(u_1)=\{x_1,x_2,x_3,u_4,v_2\}$, $N^-(u_2)=\{x_1,x_2,x_3,u_1,v_1\}$, $N^-(u_3)=\{x_3,u_1,u_2,v_1,v_2\}$,
$N^-(u_4)=\{x_3,u_2,u_3,v_1,v_2\}$,
 $N^-(v_1)=\{x_0,x_1,x_2,z,u_1\}$ and $N^-(v_2)=\{x_0,x_2,z,v_1,u_2\}$.
 
 It is not difficult to check that $T$ contains an $(x_0,x_3)$-path of length 3, but contains no $(x_0,x_3)$-path of length 4.

\section {Further terminology and notation}

If $xy$ is an arc of a digraph $D$, then we say that $x$ {\em dominates} $y$. For disjoint subsets $B$ and  $C$ of $V(D)$, let $A(B\rightarrow C)=\{xy\in A(D)\, |\, x\in B, y\in C\}$.   
We write  $B\rightarrow C$ if every vertex  of $B$  dominates every vertex of $C$. If $A\subseteq V(D)\setminus (B\cup C)$, then $A\rightarrow B\rightarrow C$ is a shortcut for $A\rightarrow B$ and $B \rightarrow C$.
If $x\in V(D)$ and $A=\{x\}$, we write $x$ instead of $\{x\}$. 

 The {\em out-neighborhood} (respectively, {\em in-neighborhood}) of a vertex $x$ is the set $N^+(x)=\{y\in V(D)\, |\, xy\in A(D)\}$ (respectively, $N^-(x)=\{y\in V(D)\, |\, yx\in A(D)\}$). Similarly, if $B\subseteq V(D)$, then $N^+(x,B)=\{y\in B \,| \, xy\in A(D)\}$ and $N^-(x,B)=\{y\in B \, |\, yx\in A(D)\}$. We define $N^{+2}(x,B)$ and $N^{-2}(x,B)$ as follows:
$$N^{+2}(x,B)= \{z\in B\setminus (\{x\}\cup N^{+}(x,B))\,\, |\,\, yz\in A(D),  y\in N^{+}(x,B)\},$$
$$N^{-2}(x,B)= \{z\in B\setminus (\{x\}\cup N^{-}(x,B))\,\, |\,\, zy\in A(D),  y\in N^{-}(x,B)\}.$$
 Note that for a vertex $x$, we have $d^+(x)=|N^+(x)|$ and $d^-(x)=|N^-(x)|$. Similarly, let $d^+(x,B)=|N^+(x,B)|$ and $d^-(x,B)=|N^-(x,B)|$.

The {\it path} (respectively, the {\it cycle}) in $D$ consisting of distinct vertices $x_1,x_2,\ldots ,x_m$ ($m\geq 2$) and arcs $x_ix_{i+1}$, $i\in [1,m-1]$  (respectively, $x_ix_{i+1}$, $i\in [1,m-1]$, and $x_mx_1$), is denoted  by $x_1x_2\cdots x_m$ (respectively, $x_1x_2\cdots x_mx_1$). The {\it length} of a cycle or a path is the number of its arcs. A {\it $k$-cycle} is a cycle of length $k$. We say that $x_1x_2\ldots x_m$  is a path from $x_1$ to $x_m$ or is an $(x_1,x_m)$-path.  The digraph $\overleftarrow{D}$ obtained from a digraph $D$ by replacing every arc $xy\in A(D)$ with the arc $yx$ is called the {\em converse} of $D.$
We will use  {\it the principle of digraph duality}: Let $D$ be a digraph and 
$\overleftarrow{D}$ the converse  of $D$. Then $D$ contains a subdigraph $H$ if and only if $\overleftarrow{D}$ contains the converse $\overleftarrow{H}$ of $H$.

\section { Preliminaries }

The following lemma states several well-known and simple claims which  are the basis of our results
and other theorems on directed cycles and paths in tournaments. The claims
will be used extensively in the proof of our result. 

\begin{lemma}\label{le1} Let $T$ be a  tournament  of order $p\geq 2$. Then the following statements are true.

(i) The tournament $T$  contains two distinct vertices $x$ and $y$ (respectively, $u$ and $v$) such that
 $d^-(x)\leq  \frac{1}{2}(p-1)$ and $d^-(y)\geq \frac{1}{2}(p-1)$ (respectively, $d^+(u)\leq \frac{1}{2} (p-1)$ and $d^+(v)\geq \frac{1}{2}(p-1)$).

(ii) If $T$ is regular, then $p=2n+1$ and for any vertex $x\in V(T)$, $d^-(x)=d^+(x)=n$. 

(iii) If $T$ is not  regular, then  $T$ contains two distinct vertices $x$ and $y$ (respectively, $u$ and $v$) such that
 $d^-(x)\leq \frac{1}{2}(p-2)$ and $d^-(y)\geq \frac{1}{2}p$ (respectively, $d^+(u)\leq \frac{1}{2}(p-2)$ and $d^+(v)\geq \frac{1}{2}p$.
 
(iv) If $T$ is almost  regular, then $p=2n$ and   $n$ vertices have indegrees equal to $n$ and the other $n$ vertices have outdegrees equal to $n$. 
 
 (v)  Let $T$ be a non-regular tournament. If  for all $v\in V(T)$, $d^-(v)< \frac{1}{2}(p+1)$ (or $d^+(v)< \frac{1}{2}(p+1)$), then $T$ is  almost regular.
 
 (vi)   Let $T$ be a tournament. If  for all $v\in V(T)$, $d^-(v)< \frac{1}{2}p$ (or $d^+(v)< \frac{1}{2}p$), then $T$ is regular. \end{lemma}

 To formulate Lemma \ref{le2},
   we need the following definition.\\
   
\noindent \textbf{Definition.} {\it By $\cal G$ we denote the set of regular tournaments, each of which has   order $6k+3\geq 9$ and vertex set $\{x,y,z\}\cup A\cup B\cup C\cup S$ with the properties $|A|=|C|= 2k-1$, $|B|=k+2$, $|S|=k$, the subtournaments induced by the subsets $A$, $C$, $\{z\}\cup B\cup S$ are regular, $A\rightarrow B\cup S\rightarrow C$, $C\rightarrow A$,  $C\rightarrow z\rightarrow A$, $x\rightarrow \{y,z\}\cup A\cup S$, $\{x,z\}\cup C\cup S\rightarrow y$, $y\rightarrow A\cup B$ and $B\cup C\rightarrow x$}.\\

Let $G\in \cal G$. Observe that $G-S$ has no $(x,y)$-path of length 3.\\

\noindent\textbf{Remark 3.} It is interesting that Thomassen \cite{[17]} used  tournaments  of the form $G-S,$ where $G\in \cal G$
to show that there are many  tournaments of order $p$ with irregularity equal to $\frac{1}{5}(p-3)$, which are not 3-strongly panconnected. \\ 

\begin{lemma}\label{le2} Let $T$ be a regular tournament of order  $2n+1$. Suppose that  $S\subseteq V(T)$ and $x$, $y$ are two distinct vertices in   $V(T)\setminus S$. Then the following hold:

(i) If $n\geq 3$ and  $k=|S|\leq \frac{1}{3}(n-1)$, then   $T-S$ contains an $(x,y)$-path of length 3, unless $T$ is isomorphic to a tournament from $\cal G$.
%\footnote{Note that if $T\in \cal G$ then $|S|=\frac{1}{3}(n-1),$ and hence if we replace $|S|\le \frac{1}{3}(n-1)$ with $|S|\le \frac{1}{3}(n-2)$, then we eliminate $\cal G$.}

(ii)  If $n\geq 5$ and  $k=|S|\leq \frac{1}{2}n$ and $T-S$ contains no $(x,y)$-path of length 3, then   $T-S$ contains an $(x,y)$-path of length 4. \end{lemma}

\begin{proof}
 Let $A=V(T)\setminus (\{x,y\}\cup S)$ and $R=N^+(x,A)\cap N^-(y,A)$. If $T-S$ contains no $(x,y)$-path of length 3, then   $|R|\leq 1$, $|A(N^+(x,A)|\geq n-k-1$ and
$$
     A(N^+(x,A)\rightarrow N^-(y,A))=\emptyset ,  \quad \hbox{i.e.,}\quad N^-(y,A)\rightarrow   N^+(x,A)  \eqno (1)
$$
and observe that
$$
|N^+(x,A)\setminus R|\geq n-k-2 \quad \hbox{and} \quad |N^-(y,A)\setminus R|\geq n-k-2.   \eqno (2)
$$ 

 (i) Suppose that $T-S$ contains no $(x,y)$-path of length 3. Note that $T'=T\langle N^+(x,A)\setminus R\rangle$ must contain a vertex $u$ with $d^-(u,T')\ge (|V(T')|-1)/2,$ the average indegree in $T'$. 
 Using (1) and (2), we obtain that the following holds for $d^-(u)$:
$$
n=d^-(u)\geq |\{x,y\}|+|N^-(y,A)|+(|N^+(x,A)\setminus R|-1)/2 $$ $$ \geq n-k+1+(n-k-3)/2= n+(n-3k-1)/2. \eqno(3)
$$
Therefore, $n-3k-1\leq 0$, i.e., $3k=n-1$ as $3k\leq n-1$, and all the inequalities that led to (3) in fact are equalities. This means that $|N^+(x,A)\setminus R|= 
 n-k-2$ (by the digraph duality, $| N^-(y,A)\setminus R|= n-k-2$), $|R|=1$, $|B|=k+2$, where $B=A\setminus (N^+(x,A)\cup N^-(y,A))$,
 $y\rightarrow B \rightarrow x$, $x\rightarrow S \rightarrow y$,
  $T\langle N^+(x,A)\setminus R\rangle$ is a regular tournament (by the digraph duality, $T\langle N^-(y,A)\setminus R\rangle$  is also a regular tournament), $N^+(x,A)\setminus R\rightarrow S\cup B \rightarrow N^-(y,A)\setminus R$, $y \rightarrow N^+(y,A)\setminus R$,      $N^-(x,A)\setminus R\rightarrow  x$, $xy\in A(T)$ and $T\langle S\cup B\cup R\rangle$ is a regular tournament. Therefore, $T$ is isomorphic to a tournament of the type $\cal G$.
  
  (ii) By contradiction, suppose that $T-S$ contains no $(x,y)$-path of lengths 3 and  4. By Lemma \ref{le1}, we know that there is a vertex $u\in N^+(x,A)\setminus R$ such that $d^+(u, N^+(x,A)\setminus R)\leq (|N^+(x,A)\setminus R|-1)/2$. Consider the out-neighbors of $u$. Since $d^+(u)=n$ and $u$ can not dominate $x$ and $y$, the number of out-neighbors of $u$ in $A$ is at least $n-k$. Further the out-neighbors of $u$ which are contained in $N^+(x,A)\setminus R$ are all in $N^{+2}(x,A)$.
  This together with (1) implies that 
$$
|N^{+2}(x,A)|\geq n-k-(|N^+(x,A)\setminus R|-1)/2=(2n-2k- |N^+(x,A)\setminus R|+1)/2. 
$$
Since $T-S$ contains no $(x,y)$-path of length 4, it follows that $N^-(y,A)\setminus R \rightarrow N^{+2}(x,A))$. This together with (1), (2) and the last inequality implies that for some vertex $v\in N^-(y,A)\setminus R$ the following holds
$$
d^+(v)\geq |\{x,y\}|+|N^+(x,A)|+(|N^{+2}(x,A)|+(|N^-(y,A)\setminus R|-1)/2  \geq 2n-2k+1.
$$
Hence, $2k\geq n+1$, which is a contradiction. This completes the proof of the lemma.\end{proof} 

 For $n=3$ and  $n=4$, it is not difficult to construct tournaments for which Lemma 3.2(ii) is not true. To see this, we consider the following examples of regular tournaments of order 7 and 9.
 
 Let $T$ be a regular tournament of order 7 (of order 9) with vertex set $V(T)=\{x,y,z\}\cup B\cup S$, where 
 $|S|=1$, $|B|=3$ (respectively, $|S|=2$, $|B|=4$). The tournament $T$ satisfies the following conditions:
 
 $A(T)$ contains the arcs $xy$, $xz$ and $zy$,
  $x\rightarrow S\rightarrow y$, $y\rightarrow B\rightarrow x$ and $T\langle \{z\}\cup B\cup S\rangle$ is a regular tournament. It is easy to check that $|S|\leq n/2$ and $T-S$ contains no $(x,y)$-path of lengths 3 and 4.
 
 Note that if we in Lemma \ref{le2}(ii), $\frac{1}{2}n$ replace with $\frac{1}{2}(n+1)$, then there is a regular tournament 
 $T$ of order $2n+1=11$ ($2n+1=15$) which contains a subset $S\subseteq V(T)$ with $|S|=k=\frac{1}{2}(n+1)$ and two distinct vertices $x,y\in V(T-S)$ such that $T-S$ contains no path of lengths 3 and 4. To see this we define a regular tournament $T$ of order $2n+1=11$ ($2n+1=15$) as follows:
 
 (a) Let $T$ be a regular tournament of order 11 with $V(T)=\{x,y,z\}\cup A\cup S$ such that $|A|=5$, $k=|S|=3$, 
 $x\rightarrow S\rightarrow y$, 
 $y\rightarrow A\rightarrow x$, the arcs $xy,xz, zy$ are in $A(T)$ and $T\langle \{z\}\cup A\cup S\rangle$ is a regular tournaments.
 
 It is easy to check that $T-S$ contains no $(x,y)$-path of length greater or equal to 3.
 
 (b) Let $T$ be a regular tournament of order 15 with $V(T)=\{x,y,z, u,v\}\cup A\cup B\cup S$ such that $|A|=|B|=3$, $k=|S|=4$, $S=\{a_1,a_2,b_1,b_2\}$, 
 the arcs $xy,xz, zy,xu, vy, vu,vz,\\ zu,a_1a_2,b_1b_2, a_2b_2,a_2b_1, a_1b_1, b_2a_1$ are in $A(T)$, $x\rightarrow S\rightarrow y$,
$\{y,z,u,v\} \rightarrow A$, 
$B\rightarrow \{x,u,v,z\}$, 
$y\rightarrow \{u\} \cup B$,
$A\rightarrow \{x\}\cup B$,
$v\rightarrow \{x\}\cup A$,
$S\rightarrow v$,
$u\rightarrow  A\cup S$, the induced subtournaments 
$T\langle A\rangle$ and $T\langle B\rangle$ are regular tournaments, 
$\{b_1,b_2\}\rightarrow A\rightarrow \{a_1,a_2\}$,
$\{a_1,a_2\}\rightarrow B\rightarrow \{b_1,b_2\}$ and
$\{b_1,b_2\}\rightarrow z\rightarrow \{a_1,a_2\}$. 

It is not difficult to check that $T-S$ contains no $(x,y)$-path of length 3 and 4. Note that $T-S$ contains an $(x,y)$-path of every length $5,6,\ldots , 10$.

In Lemmas \ref{le3} and \ref{le4}, we suppose that $P:= x_0x_1\ldots x_r$ is an $(x_0,x_r)$-path
of length $r$ in a tournament $T$ and $z$ is a vertex in $V(T)\setminus V(P)$ such  that $
\{x_{\alpha +1}, x_{\alpha +2},\ldots , x_r\}\rightarrow z \rightarrow \{x_0,x_1,\ldots , x_{\alpha}\}$, where $\alpha \in [2,r-3]$.  Moreover, any $(x_0, x_r)$-path of length $r + 1$ with vertex set $\{z\}\cup  V (P)$ is
denoted by $Q,$ and we assume that $T$ contains no such a path $Q$.
%Moreover, by $Q$ is denoted any $(x_0,x_r)$-path of length $r+1$ with vertex set $\{z\}\cup V(P)$, and is assumed that $T$ contains no such $Q$ path.\\ 

\begin{lemma}\label{le3} Suppose that $x_sx_t\in A(T)$ with
 $s\in [1,\alpha-1]$ and $t\in [\alpha+3,r]$. Then $A(\{x_0,x_1,\ldots , x_{s-2}\}\rightarrow \{x_{\alpha+2}, x_{\alpha+3}, \ldots , x_{t-1}\})=\emptyset$ when $s\geq 3$, and  $$A( x_{s-1}\rightarrow \{x_{\alpha+2}, x_{\alpha+3}, \ldots , x_{t-1}\})=\emptyset, \quad  when \quad t-s\not=5.$$\end{lemma}
 
\begin{proof} By contradiction, suppose that there exist integers $a\in [0,s-1]$ and $b\in [\alpha+2,t-1]$ such that $x_ax_b\in A(T)$. Observe that $\{x_{b-1},x_{t-2},x_{t-1}\}\rightarrow z \rightarrow \{x_{a+1},x_{a+2},x_{s+1}\}$. Note that $t-s\geq 4$. Now we prove the following facts.

{\it Fact 3.1}. (i)  $x_{a+1}x_{b-1}\in A(T)$.

(ii) If $a\leq s-2$, then $x_{s+1}x_{a+1}\in A(T)$ and if $a=s-1$, then $x_{s+2}x_s\in A(T)$ and $t-s\geq 6$.

%(i)  $x_{a+1}x_{b-1}\in A(T)$ and if $a\leq s-2$, then $x_{s+1}x_{a+1}\in A(T)$.

%(ii) If $a=s-1$, then $x_{s+2}x_s\in A(T)$ and $t-s\geq 6$.

{\it Proof}. (i). Indeed, if $x_{b-1}x_{a+1}\in A(T)$, then $Q=x_0x_1\ldots x_ax_b\ldots x_{t-1}zx_{s+1}\ldots x_{b-1}\\x_{a+1}\ldots x_sx_t\ldots x_r$, a contradiction.

(ii).  If $a\leq s-2$ and  $x_{a+1}x_{s+1}\in A(T)$, then  $$Q=x_0x_1\ldots x_{a+1}x_{s+1}\ldots x_{t-1}zx_{a+2}\ldots x_{s}x_{t}\ldots  x_r.$$ If $a=s-1$ and $x_sx_{s+2}\in A(T)$. Then either $x_{s+1}x_{t-1}\in A(T)$ and  $$Q=x_0x_1\ldots  x_sx_{s+2}\ldots x_{t-2}zx_{s+1}x_{t-1}  \ldots  x_r$$ or $x_{t-1}x_{s+1}\in A(T)$ and $Q=x_0\ldots x_{s-1}x_b\ldots x_{t-1}x_{s+1}\ldots  x_{b-1}zx_sx_t\ldots x_r$. Thus, in both cases we have a contradiction.

 It is easy to see that from  $x_sx_{b-1}\in A(T)$ and $x_{s+2}x_s\in A(T)$ it follows  that $b-1\geq s+3$ and hence, $t-s\geq 6$ since $t-s\not=5$ when $a=s-1$.

{\it Fact 3.2.} If $x_{i}x_{a+1}\in A(T)$ with $i\in [s+1,b-3]$, then $x_{i+2}x_{a+1}\in A(T)$.
 
 {\it Proof}. To prove it by contradiction, suppose that $x_{i}x_{a+1}\in A(T)$  with $i\in [s+1,b-3]$  and $x_{a+1}x_{i+2}\in A(T)$.  If $x_{t-1}x_{i+1}\in A(T)$, then $$Q=x_0x_1\ldots x_{a}x_b\ldots x_{t-1}x_{i+1}\ldots x_{b-1}\\zx_{s+1}\ldots x_{i}x_{a+1}\ldots x_sx_t\ldots x_r,$$ and if $x_{i+1}x_{t-1}\in A(T)$, then $Q=x_0x_1\ldots x_{a+1}x_{i+2}\ldots x_{t-2}zx_{a+2}\ldots x_{i+1}x_{t-1}\ldots x_r$,   a contradiction. 
 
 Using Fact 3.2, it is not difficult to see that there is no $i\in [s+1,b-3]$ such that $\{x_i,x_{i+1}\}\rightarrow x_{a+1}$, i.e., $d^-(x_{a+1},\{x_i,x_{i+1}\})\leq 1$ (for otherwise, we obtain that $x_{b-1}x_{a+1}\in A(T)$, contradicting Fact 3.1 that  $x_{a+1}x_{b-1}\in A(T)$). By Fact 3.1 we have that if $a\leq s-2$, then $x_{s+1}x_{a+1}\in A(T)$, and  $a= s-1$, then $x_{s+2}x_{s}\in A(T)$.  This together with  $d^-(x_{a+1},\{x_i,x_{i+1}\})\leq 1$ and  Fact 3.2 implies that\\ 
$$
\hbox{if} \,\, a\leq s-2, \,\, \hbox{then}\,\, \{x_{s+1}, x_{s+3},\ldots , x_{b-2}\}\rightarrow x_{a+1}\rightarrow \{x_{s+2}, x_{s+4},\ldots , x_{b-1}\},
$$
$$
\hbox{if} \,\, a=s-1, \,\, \hbox{then}\,\, \{x_{s+2}, x_{s+4},\ldots , x_{b-2}\}\rightarrow x_{s}\rightarrow \{x_{s+1}, x_{s+3},\ldots , x_{b-1}\}. \eqno (4)
$$
Thus, in both cases we have $x_{b-2}x_{a+1}\in A(T)$. Now using (4), we obtain, 
if  $x_{t-1}x_{b-1}\in A(T)$, then $b\leq t-2$ and $Q=x_0\ldots x_ax_b\ldots x_{t-1}x_{b-1}zx_{s+1}\ldots x_{b-2} x_{a+1}\ldots  x_sx_t\ldots x_r$, and if $b\leq t-2$ and  $x_{b-1}x_{t-1}\in A(T)$, then $Q=x_0\ldots x_ax_b \ldots x_{t-2}zx_{a+1}\ldots x_{b-1} x_{t-1}\ldots x_r$, thus in both cases we have a contradiction.  
We may therefore assume that $b=t-1$. In this case, if $x_bx_{s+1}\in A(T)$, then $Q=x_0x_1\ldots x_ax_bx_{s+1}\ldots x_{b-1}zx_{a+1}\ldots  x_sx_t\ldots x_r$, a contradiction. Therefore, we may assume that $x_{s+1}x_b\in A(T)$. This follows that if $a\leq s-2$, then   $Q=x_0\ldots x_{a+1}x_{s+2}\ldots x_{b-1}zx_{a+2}\ldots x_{s+1}x_b\ldots x_r$, a contradiction. 
Assume that $a=s-1$. Then by Fact 3.1(ii), $t-s\geq 6$. 
 Let $t\geq \alpha+4$. It is easy to see that if $x_{b-1}x_{s+2}\in A(T)$, then by (4), $x_{a+1}x_{s+2}\in A(T)$ and $Q=x_0x_1\ldots x_sx_{b-1}x_{s+2}\ldots x_{b-2}zx_{s+1}x_b\ldots x_r$, and if $x_{s+2}x_{b-1}\in A(T)$, then $Q=x_0x_1\ldots x_sx_{s+3}\ldots x_{b-2}zx_{s+1}x_{s+2}x_{b-1}\ldots x_r$. 
 Let now
$t= \alpha+3$. Then $b=\alpha+2$ and $z\rightarrow \{x_{s+1},x_{s+2},\ldots , x_{\alpha}\}$. From $t-s\geq 6$ and (4)  it follows that $t-s\geq 7$ and $\alpha \geq s+4$. From (4) we also have that $x_{b-4}x_s\in A(T)$. 
Therefore,
 if $x_{b}x_{b-3}\in A(T)$, then $Q=x_0x_1\ldots x_{s-1}x_bx_{b-3}x_{b-2}x_{b-1}zx_{s+1}\ldots x_{b-4}x_sx_t\ldots x_r$, a contradiction. We may therefore assume that $x_{b-3}x_{b}\in A(T)$. If $x_{b-2}x_{s+1}\in A(T)$, then  $Q=x_0x_1\ldots x_{s}x_{b-1}zx_{b-2}x_{s+1}\ldots x_{b-3}x_b\ldots x_r$, a contradiction.  Thus, we have that $x_{b-3}x_{b}\in A(T)$  and $x_{s+1}x_{b-2}\in A(T)$. Therefore,  $Q=x_0x_1\ldots x_{s+1}x_{b-2}x_{b-1}zx_{s+2}\ldots x_{b-3}x_b\ldots x_r$, a contradiction. 
Thus, in all cases we have a contradiction.  \end{proof}

\begin{lemma}\label{le4} Suppose that 
  $x_sx_t\in A(T)$ with  $s\in [\alpha , r-2]$ and  $t\in [s+2,r]$. If $k=\lfloor \frac{1}{2}(t-s)\rfloor$, then 
$A(\{x_{0}, x_{1},\ldots , x_{\alpha -1}\}\rightarrow 
\{x_{s+1}, x_{s+2},\ldots , x_{s+k}\})=\emptyset$.\end{lemma}

\begin{proof} The proof by induction on $k= \lfloor \frac{1}{2}(t-s)\rfloor$.  Observe that $\{x_{s+1}, x_{t-1}\} \rightarrow z$. For the base step, it is easy to see that if $x_{i}x_{s+1}\in A(T)$ with $i\in [0,\alpha -1 ]$, then $Q= x_0x_1\ldots  x_i x_{s+1}\ldots x_{t-1}zx_{i+1}\ldots x_sx_t\ldots x_r$, a contradiction. We may therefore assume that
$A(\{x_{0}, x_{1},\ldots , x_{\alpha -1}\}\rightarrow 
x_{s+1})=\emptyset$. This means that
 if $2\leq t-s\leq 3$, then  the lemma is true. Assume that $t-s\geq 4$. For the inductive step, we assume that  if $x_{s_1}x_{t_1}\in A(T)$ with $s_1\in [\alpha , r-2]$,  $t_1\in [s_1+2,r]$ and $t_1-s_1 < t-s$, then 
$A(\{x_{0}, x_{1},\ldots , x_{\alpha -1}\}\rightarrow 
\{x_{s_1+1}, x_{s_1+2},\ldots,  x_{s_1+k_1}\})=\emptyset$, where  $k_1=\lfloor \frac{1}{2}(t_1-s_1)\rfloor$.  

If $x_{t-1}x_{s+1}\in A(T)$, then for all $i\in [1,\alpha -1]$ and $j \in [s+2,t-1]$, 
$x_{i}x_{j}\notin    A(T)$ and we are done (for otherwise, if $x_{i}x_{j}\in    A(T)$, then $Q= x_0x_1\ldots  x_i x_{j}\ldots x_{t-1}x_{s+1}\ldots x_{j-1}z\\x_{i+1}\ldots x_sx_t \ldots x_r$, a contradiction). Therefore, $A(\{x_0,x_1,\ldots , x_{\alpha -1}\rightarrow \{x_{s+1},x_{s+2},\ldots ,\\ x_{t-1}\})=\emptyset$ and we are done.
Now assume  that $x_{s+1}x_{t-1}\in A(T)$. Then $t-1-(s+1)<t-s$ and, by the induction hypothesis,  $A(\{x_{0}, x_{1},\ldots , x_{\alpha -1}\}\rightarrow 
\{x_{s+2}, \ldots , x_{s+1+m}\})=\emptyset$, where $m=\lfloor (t-s-2)/2\rfloor$. Thus, $A(\{x_{0}, x_{1},\ldots , x_{\alpha -1}\}\rightarrow 
\{x_{s+1}, x_{s+2},\ldots , x_{s+1+m}\})=\emptyset$. This implies that Lemma \ref{le4} is true since $m+1=\lfloor \frac{1}{2}(t-s)\rfloor$.  \end{proof}

\section {Proof of the main result}

For convenience of the reader, we restate it here.\\

\noindent\textbf{Theorem \ref{th4}.} {\it Let $T$ be a regular tournament of order $2n+1\geq 11$ and let $S$ be a subset  in $V(T)$. Let  $|S|\leq  \frac{1}{2}(n-2)$ and $x,y$ 
 be two distinct vertices in $V(T)\setminus S$. If 
   $T- S$ contains an $(x,y)$-path of length $r$, where $r\in [3, 2n-|S|-1]$, then 
    $T- S$ also contains an $(x,y)$-path of length $r+1$}.
   
\begin{proof} Recall that $P = x_0x_1 \dots  x_r$ is a path of length $r$ in $T$
and $k=|S|.$

Observe that for any vertex $x\in V(T-S)$, $d^-(x,V(T-S))\geq n-k$ and $d^+(x,V(T-S))\geq n-k$. Now by  $Q$ we denote any  $(x_0,x_r)$-path of length $r+1$ in $T- S$.   Suppose that $T-S$ has no such path $Q$.  
Let $A=V(T)\setminus (V(P)\cup S)$. We will prove a series of claims (Claims 1-13).

\textbf{Claim 1:} $N^-(x_r, A)=N^+(x_0, A)=\emptyset$, i.e., 
$d^-(x_r,A)= d^+(x_0,A)=0$.

\textbf{Proof:} By the digraph duality, it suffices to prove that  $N^-(x_r, A)=\emptyset$. Suppose, on the contrary, that  $N^-(x_r, A)\not=\emptyset$. Let $z\in N^-(x_r, A)$, i.e., $zx_r\in A(T)$. Then $z\rightarrow \{x_0,x_1, \ldots , x_r\}$,  $r\leq n-1$ and $|N^-(z,A)|\geq n-k$. It is easy to see that
$$
A(\{x_0,x_1, \ldots , x_{r-2}\}\rightarrow N^-(z,A)\cup \{z\})=\emptyset.  \eqno (5)
$$
We distinguish the following two cases depending on $r$.

\textbf{Case 1.1.} $r\geq n-k$.
 We know that $N^-(z,A)$ contains a vertex $u$ such that $d^-(u,N^-(z,A))\leq 0.5(d^-(z,A)-1)$. This together with (5) implies that 
$$ 
|N^{-2}(z,A)|\geq |N^{-2}(z,A)\cap N^-(u)|\geq n-k-2-0.5(d^-(z,A)-1).    \eqno (6)
$$
It is easy to see that $A(\{x_0,x_1, \ldots , x_{r-3}\}\rightarrow N^{-2}(z,A))=\emptyset$. This together with (5) implies that $\{z\} \cup N^-(z,A)\cup N^{-2}(z,A) \rightarrow \{x_0,x_1, \ldots , x_{r-3}\}$. Therefore, $\{x_0,x_1, \ldots , x_{r-3}\}$ contains a vertex $v$ such that $d^-(v)\geq |N^-(z,A)|+ |N^{-2}(z,A)|+1+ 0.5(r-3)$ and unless $\{x_0,x_1,\dots ,x_{r-3}\}$ induces a regular tournament, we can find a vertex $v$ so that equality does not hold. Now using (6), $r\geq n-k$ and the fact that $|N^-(z,A)|\geq n-k$, we obtain
$$n=d^-(v)\geq d^-(z,A)+n-k-2-0.5(d^-(z,A)-1)+0.5(r-1)\geq 2n-2k-2.$$
 Therefore, $2k=n-2$ (as $2k\leq n-2$) and all inequalities which were used in
the last inequality in fact are equalities, in particular, $T\langle\{x_0,x_1,\ldots , x_{r-3}\}\rangle$ is a regular tournament.  Therefore, $\{x_0,x_1,\ldots , x_{r-3}\}\rightarrow  x_{r-2}$. Hence, 
 \begin{eqnarray*}
  d^-(x_{r-2}) & \geq & |N^-(z,A)|+|\{z\}|+|\{x_0,x_1,\ldots , x_{r-3}\}|\\
  & \ge & n-k+1+r-2\\
  &\ge & n-k+1+n-k-2\\
  &=& 2n-2k-1\\
  &\ge& 2n -(n-2)-1\\
  &=& n+1,
 \end{eqnarray*} 
 a contradiction.
 
 %$d^-(x_{r-2})\geq |N^-(z,A)|+|\{z\}|+|\{x_0,x_1,\ldots , x_{r-3}\}|= n+1$, 

\textbf{Case 1.2.} $r\leq n-k-1$. Then $N^+(x_0,A)\not= \emptyset$. It is easy to see that   
$$
N^+(x_0,A)\cap N^-(z,A)= A(N^+(x_0,A)\rightarrow  N^-(z,A)\cup \{z,x_1,x_2,\ldots ,x_r\})=\emptyset.  \eqno (7) 
$$
We know that $N^+(x_0,A)$ contains a vertex $u$ such that $d^+(u,N^+(x_0,A)\leq 0.5(d^+(x_0,A)-1)$. From this and (7) it follows that
$|N^{+2}(x_0,A)|\geq n-k-0.5(d^+(x_0,A)-1)$.  It is not difficult to see that $N^-(x_r,A)\rightarrow  \{x_0,x_1,\ldots ,x_r\}\cup N^+(x_0,A) \cup N^{+2}(x_0,A)$. Therefore, $N^-(x_r,A)$ contains a vertex $y$ such that 
$$
n=d^+(y)\geq |N^+(x_0,A)| + |N^{+2}(x_0,A)|+ 0.5(|N^-(x_r,A)|-1)+r+1.
 $$
Now, since $\min\{|N^+(x_0,A)|, |N^-(x_r,A)|\}\geq n-k-r$,  we obtain that $n=d^+(y)\geq 2n-2k+1$. Hence, $2k\geq n+1$, which contradicts that $2k\leq n-2$. Claim 1 is proved. \fbox \\\\

Using Claim 1, $n\geq 5$ and $2k\leq n-2$, we obtain   that $|A|\leq n$, $r=2n-k-|A| \geq n-k\geq 3.5$ (i.e., $r\geq 4$) and 
 $$
 \min\{|N^+(x_0,V(P))|, |N^-(x_r,V(P))|\}\geq n-k. \eqno (8)
 $$
 
 \textbf{Claim 2:} $N^+(x_1,A)= N^-(x_{r-1},A)=\emptyset$.
 
 \textbf{Proof:} By the digraph duality, it suffices to prove that  $N^-(x_{r-1}, A)=\emptyset$. Suppose, on the contrary, that $N^-(x_{r-1}, A)\not=\emptyset$. Let $z$ be a vertex in $N^-(x_{r-1}, A)$. Then  $zx_{r-1}\in A(T)$, $z\rightarrow \{x_0,x_1,\ldots , x_{r-1}\}$ and   $|N^-(z, A)|\geq n-k-1$. 
 
\textbf{Case 2.1.} $N^-(z, A)\rightarrow x_{r-1}$. Then
$$
A(\{x_0,x_1, \ldots , x_{r-3}\}\rightarrow \{z\} \cup N^-(z,A)\cup N^{-2}(z,A))=\emptyset.  \eqno (9)
$$
Since $N^-(z,A)$ contains a vertex $u$ such that
 $d^-(u,N^-(z,A))\leq 0.5(|N^-(z,A)|-1)$,  it follows that $|N^{-2}(z,A)|\geq n-k-1-0.5(|N^-(z,A)|-1)$. Now using (9) and $r\geq n-k$, we obtain that  $\{x_0,x_1, \ldots , x_{r-3}\}$ contains a vertex $v$ such that
    $$
    n=d^-(v)\geq 0.5(r-3)+n-k-1-0.5(|N^-(z,A)|-1)+|N^-(z,A)|+1\geq 2n-2k-1.5,
    $$ a contradiction to $2k\leq n-2$.
 
\textbf{Case 2.2.} $A(x_{r-1} \rightarrow  N^-(z, A))\not=\emptyset$. Since $\{z\}\cup N^-(z, A)\subseteq A$ and  $|N^-(z, A)|\geq n-k-1$, we have that $|A|\geq n-k$. Observe that 
$$
 N^-(z,A)\cup N^-(x_{r-1},A)\rightarrow \{x_0,x_1, \ldots , x_{r-3}\}.  \eqno (10)
$$
Hence, $\{z\}\cup N^-(z, A)\rightarrow x_1$, as $r\geq 4$. Therefore, for some $y\in N^-(z, A)$, $x_{r-1}\rightarrow  y\rightarrow x_1$. If $x_ix_r\in A(T)$ with $i\in [1,r-2]$ (respectively, with $i\in [1,r-3]$), then
 $x_0x_{i+1}\notin A(T)$ (respectively, $x_0x_{i+2}\notin A(T)$), for otherwise 
 $Q= x_0x_{i+1}\ldots x_{r-1}yx_1\ldots x_ix_r$ (respectively, $Q= x_0x_{i+2}\ldots x_{r-1}yzx_1\ldots x_ix_r$), which is a contradiction. From this and (8) we have, $$d^-(x_0,\{x_2,x_3,\ldots , x_{r-1}\})\geq n-k-2. \eqno (11)
 $$
 This together with $|A|\geq n-k$ and $A\rightarrow x_0$ implies that $n=d^-(x_0)\geq |A|+d^-(x_0,\{x_2,\\x_3,\ldots , x_{r-1}\})\geq 2n-2k-2$, which in turn implies $2k=n-2$, $|A|=n-k$ (i.e., $A=\{z\}\cup N^-(z, A)$), $r=n\geq 6$, $k\geq 2$ (as $n\geq 5$ and $n$ is even and $r=2n-k-|A|$) and $d^-(x_0,\{x_2,x_3,\ldots , x_{r-1}\})=n-k-2$. Now by the above
arguments, it is not difficult to see that  $N^-(x_r,\{x_1,x_2,\ldots , x_{r-2}\})= 
\{x_{k+1},x_{k+2},\ldots , x_{n-2}\}$.
Since $yx_2\in A(T)$ (by (10) and $r\geq 6$), it follows that if $x_1x_{i+1}\in A(T)$ with $i\in [k+1, n-2]$, then $Q=x_0x_1x_{i+1}\ldots x_{n-1}yx_2\ldots x_ix_n$, a contradiction. We may therefore assume that $d^+(x_1, \{x_{k+2},x_{k+3},\ldots , x_{n-1}\})=0$. This together with $d^+(x_1, \{x_0\}\cup A)=0$ implies that $d^-(x_1)\geq |A|+n-k-1\geq 2n-2k-1$, a contradiction. Claim 2 is proved. \fbox \\\\

From Claim 2 it follows that $d^-(x_{r-1},\{x_0,x_1,\ldots , x_{r-2}\})\geq n-k$. Hence, $r\geq n-k+1$.  Using (11) and Claim 1, we obtain, $d^-(x_0)\geq n-k-2+|A|$. Therefore, $|A|\leq k+2$.

 \textbf{Claim 3:} $|A|\leq k+1$.
 
 \textbf{Proof:} The proof is by contradiction. Suppose that   $|A|\geq k+2$. Then,  $|A|=k+2\geq 2$ and $r=2n-2k-2$ since $|A|\leq k+2$.  This implies that $r$ is even and $r\geq n\geq  6$ since $n\geq 5$ and $2k\leq n-2$. Since $x_{r-1}\rightarrow A\rightarrow x_1$ (Claim 2), it is not difficult to see that if $x_ix_r\in A(T)$ with $i\in [1,r-2]$, then $x_0x_{i+1}\notin A(T)$ (for otherwise, $Q=x_0x_{i+1}\ldots x_{r-1}zx_1\ldots x_ix_r$, where $z\in  A$). This together with $A\rightarrow x_0$, $d^-(x_r, \{x_1,x_2,\ldots , x_{r-2}\})\geq n-k-2$ and $|A|=k+2$ implies that $n=d^-(x_0)\geq |A|+
 d^-(x_r,\{x_1,x_2, \ldots , x_{r-2}\})\geq k+2+n-k-2=n$. This means that $d^-(x_r,\{x_1,x_2, \ldots , x_{r-2}\}=n-k-2$ and the out-neighbors of $x_0$ in $\{x_2,x_3,\ldots , x_{r-1}\}$ are only those vertices $x_i$ for which $x_{i-1}x_r\in A(T)$. From this we conclude that there is no $i\in [1,r-3]$ such that $x_ix_r\in A(T)$ and $x_{i+1}x_r\notin A(T)$ since in the converse case $x_0x_{i+2}\in A(T)$ and $Q=x_0x_{i+2}\ldots x_{r-1}a_1a_2zx_1\ldots x_ix_r$, where $a_1,a_2\in  A$ and $a_1a_2\in A(T)$. Therefore, 
 $N^-(x_r, \{x_1,x_2,\ldots , x_{r-2}\})=\{x_{n-k-1},x_{n-k},\ldots , x_{r-2}\} $  and 
$N^+(x_0, 
\{x_2,x_3,\ldots , x_{r-1}\})= \{x_2,x_3,\ldots , x_{n-k-1}\}$. In particular, $x_rx_1\in A(T)$ as $n\geq 6$ and $n-k\geq 3$.

Assume first that for some $y\in A$, $yx_2\in A(T)$.  Using the above equalities, we obtain $
\{x_{n-k},x_{n-k+1},\ldots , x_{r-1}\}\rightarrow x_1$ (for otherwise, $Q= x_0x_1x_j\ldots  x_{r-1}yx_2\ldots x_{j-1}x_r$, where $j\in [n-k, r-1]$). Now, since $A\cup \{x_0,x_r\}\rightarrow x_1$ and $|A|=k+2$, we obtain $d^-(x_1)\geq n+2$, a contradiction.

Assume next that $A(A\rightarrow x_2)=\emptyset$, i.e., $x_2\rightarrow A$. Then $\{x_2,x_3,\ldots , x_{r}\}\rightarrow A$ and there is a vertex $u\in A$ such that $d^-(u)\geq 2n-2k-3+ (k+1)/2$. Therefore, $2k=n-2$ as 
$2k\leq n-2$. Thus $k\leq 1$, which is a contradiction, since $2k=n-2$ and $n\geq 6$. Claim 3 is proved. \fbox \\\\

Since $r=2n-k-|A|$, $|A|\leq k+1$ and $2k\leq n-2$, we have  $r\geq n+1 \geq 6$ as $n\geq 5$.\\

\textbf{Claim 4:} $N^+(x_2,A)= N^-(x_{r-2},A)=\emptyset$.
 
 \textbf{Proof:} By the digraph duality, it suffices to prove that  $N^+(x_{2}, A)=\emptyset$. Suppose, on the contrary, that $N^+(x_{2}, A)\not=\emptyset$. 
 Let $z$ be a vertex in $N^+(x_{2}, A)$. Then $\{x_2,x_3,\ldots , x_r\}\rightarrow z$.
 This together with $z\rightarrow \{x_0,x_1\}$ implies that $|N^+(z, A)|\geq n-k-2$  and  $|A|\geq n-k-1$. Now using the facts that $|A|\leq k+1$ (Claim 3) and $2k\leq n-2$, we obtain that
$2k=n-2$, $|A|=k+1$ and $r=n+1$. From $2k=n-2$ it follows that $n$ is even, $n\geq 6$ (as $n\geq 5$) and $k\geq 2$.
Since $\{x_0\}\cup A\rightarrow x_1$ and $r=n+1$, it follows that there is a vertex $x_j\in \{x_4,x_5, \ldots , x_{n+1}\}$ such that $x_1x_j\in A(T)$. Therefore, if 
  $x_0x_2\in A(T)$, then  $Q=x_0x_2\ldots x_{j-1}zx_1x_j\ldots x_r$, and if 
 $x_0x_3\in A(T)$, then for some $y\in N^+(z,A)$ we have $Q=x_0x_3\ldots x_{j-1}zyx_1x_j\ldots x_r$,  which is a contradiction. 
 We may therefore assume that $\{x_2,x_3\}\rightarrow x_0$. Since $\{x_2,x_3\}\cup  A\rightarrow x_0$, $\{x_2,x_3,\ldots , x_r\} \rightarrow z$ and $r=n+1$, it follows that   $d^+(x_0,\{x_4,x_5,\ldots , x_{n}\})\geq n-k-2\geq 2$ and $d^+(z,A)=k$. Let $x_0x_m\in A(T)$ with $m\in [4,n]$ and $m$ is the minimum with these properties. Since $A\rightarrow x_1$, $x_{n}z\in A(T)$ and $T\langle A\rangle$ contains a path $zyu$ of length two,  it is easy to see that  $d^-(x_{n+1},\{x_{m-3},x_{m-2},x_{m-1}\})=0$ and if $x_0x_j\in A(T)$ with $j\in [4,n]$, then $x_{j-1}x_{n+1}\notin A(T)$. These together with $d^-(x_{n+1},A)=0$ and $|A|=k+1$ imply that $d^+(x_{n+1})\geq |A|+n-k-2+2= n+1$, a contradiction. Claim 4 is proved. \fbox \\\\
 
 From Claims 1, 2 and  4 it follows that $\{x_{r},x_{r-1},x_{r-2}\}\rightarrow A \rightarrow \{x_0,x_1,x_2\}$.  
 
 Let $\alpha :=\max\{i\in [2,r-3]\, |\, A\rightarrow x_{i}\}$. From this and the assumption that $T$ contains no $Q$ path it follows that 
 $A\rightarrow \{x_0,x_1,\ldots , x_{\alpha}\}$, $ A(x_j\rightarrow A)\not= \emptyset$ for all $j\in [\alpha +1,r]$ 
 and $A$ contains a vertex $z$ such that
 $$
\{x_{\alpha +1}, x_{\alpha +2},\ldots , x_r\}\rightarrow z \rightarrow \{x_0,x_1,\ldots , x_{\alpha}\}. \eqno (12)
$$

 \textbf{Fact 1:} (i) If $2k=n-2$, then $n$ is even (i.e., $n\geq 6$, as $n\geq 5$), $n-k-|A|\geq 1$ and $r\geq n+1\geq 7$. 
 
 (ii) If $2k\leq n-3$, then $n-k-|A|\geq 2$ and $r\geq n+2\geq 7$.
 
\textbf{Proof:} Indeed, $n-k-|A|\geq n-2k-1$ and $r=2n-k-|A|\geq 2n-2k-1$ as $|A|\leq k+1$. Now, if $2k=n-2$, then $n-k-|A|\geq 1$ and $r\geq n+1\geq 7$, and if $2k\leq n-3$, then $n-k-|A|\geq 2$ and $r\geq n+2\geq 7$. \fbox \\
 
 \textbf{Fact 2:}  $n-k-|A|\leq \alpha \leq n-1- \lfloor \frac{1}{2}|A|\rfloor \leq n-1$.

 \textbf{Proof:} Since  $\{x_{\alpha +1},x_{\alpha +2},\ldots ,x_r\}\rightarrow z$, we have $n=d^-(z)\geq r-\alpha =2n-k-|A|-\alpha$ and $\alpha \geq n-k-|A|$. From $A\rightarrow \{x_0,x_1, \ldots, x_{\alpha}\}$ it follows that  there exists a vertex $x\in A$ such that $n=d^+(x)\geq \alpha +1+\lfloor \frac{1}{2}|A|\rfloor$. Hence, $\alpha\leq n-1- \lfloor \frac{1}{2}|A|\rfloor \leq n-1$. \fbox \\\\

Let $B:=\{x_i\mid i\in [1,\alpha -1], \,  \, A(x_i\rightarrow \{x_{\alpha +2}, \ldots , x_r\})\not= \emptyset\}$.

 \textbf{Proposition 1:} (i) If $x_i\in B$, then 
 $A(\{x_0,x_1,\ldots ,  x_{i-1}\}\rightarrow x_{i+1})= \emptyset$. 
 
 (ii) If $x_jx_r\in A(T)$ with $j\in [\alpha,r-2]$, then 
 $A(\{x_0,x_1,\ldots ,  x_{\alpha -1}\}\rightarrow x_{j+1})= \emptyset$.

\textbf{Proof:} (i) $x_i\in B$ means that
 $i\in [1,\alpha -1]$ and there is an integer $b\in  [\alpha +2, r]$ such that $x_ix_b\in A(T)$. Assume that for some $a\in [0,i-1]$, $x_ax_{i+1}\in A(T)$. By (12), $zx_{a+1}, x_{b-1}z\in A(T)$. Therefore, $Q=x_0\ldots x_ax_{i+1}\ldots x_{b-1}zx_{a+1}\ldots x_ix_b\ldots x_r$, a contradiction. So, $A(\{x_0,x_1,\ldots ,  x_{i-1}\}\rightarrow x_{i+1})= \emptyset$. 
 
 (ii) Now assume that $x_jx_r\in A(T)$ with $j\in [\alpha,r-2]$ and  $x_ix_{j+1}\in A(T)$ with $i\in [0,\alpha -1]$. Again using (12), we obtain $Q=x_0x_1\ldots x_ix_{j+1}\ldots x_{r-1}zx_{i+1}\ldots x_j x_r$, a contradiction. \fbox \\\\
 
 Now, we divide the proof of Theorem \ref{th4} into two cases. Note that, by digraph duality and Lemma \ref{le3}, the second case easily follows from the first case.\\
 
\textbf{Case I.} $A(\{x_1,x_2,\ldots ,  x_{\alpha -1}\}\rightarrow x_{r})= \emptyset$, i.e., $x_r\rightarrow \{x_1,x_2,\ldots ,  x_{\alpha -1}\}$.

For this case first we need to show the following claims below (Claims 5-13).

\textbf{Claim 5:} $B\not= \{x_1,x_2,\ldots ,  x_{\alpha -1}\}$.

\textbf{Proof:} Suppose, on the contrary, that 
 $B= \{x_1,x_2,\ldots ,  x_{\alpha -1}\}$. Let $\alpha \geq 3$. Then using Proposition 1(i), we see that $d^+(x_1,\{x_0,x_3,\ldots ,  x_{\alpha },x_r\})=0$.
From $d^-(x_r, V(P))\geq n-k$ and $x_r\rightarrow \{x_1,x_2,\ldots ,  x_{\alpha -1}\}$ it follows that $d^-(x_r, \{x_{\alpha},x_{\alpha +1},\ldots , x_{r-2}\})\geq n-k-2$. 

By running $j$ from $\alpha$ to $r-2$ in Proposition 1(ii), we have $$d^-(x_r,\{x_{\alpha},x_{\alpha +1}, \ldots , x_{r-2}\}) +
d^+(x_1,\{x_{\alpha +1},x_{\alpha +2}, \ldots , x_{r-1}\})\le r-1-\alpha.$$ Therefore,
\begin{eqnarray*}
d^-(x_1,\{x_{\alpha+1},x_{\alpha +2}, \ldots , x_{r-1}\})&=& r-1-\alpha - d^-(x_1,\{x_{\alpha+1},x_{\alpha +2}, \ldots , x_{r-1}\})\\
&\ge & d^-(x_r,\{x_{\alpha},x_{\alpha +1}, \ldots , x_{r-2}\})\\
&\ge & n-k-2.
\end{eqnarray*}
Hence, by this inequality and $\alpha\geq n-k-|A|$ (Fact 2) we have 
 $$
 n=d^-(x_1)\geq |A|+|\{x_0,x_3,\ldots ,  x_{\alpha},x_r\}|+
 d^-(x_1,\{x_{\alpha +1},x_{\alpha +2}, \ldots , x_{r-2}\})$$ $$\geq |A|+\alpha+n-k-2 \geq 2n-2k-2. 
$$
Let now $\alpha=2$. From $2=\alpha \geq n-k-|A|$ it follows that  $|A|\geq n-k-2$. Therefore,
$$
 n=d^-(x_1)\geq |A|+|\{x_0,x_r\}|+
 d^-(x_1,\{x_{\alpha +1},x_{\alpha +2}, \ldots , x_{r-2}\})$$ $$\geq |A|+2+n-k-2 \geq 2n-2k-2. 
$$
Since $k=|S|\le (n-2)/2$, in both cases we have equalities, that is, $2k=n-2,$
$k\geq 2$, $\alpha =n-k-|A|$ and $d^-(x_r,\{x_{\alpha},x_{\alpha +1},\ldots , x_{r-2}\})=n-k-2$. From these we obtain that  $|\{x_{\alpha +1},x_{\alpha +2},\ldots , x_{r}\}|=r-\alpha =2n-k-|A|-(n-k-|A|)=n$ and if $x_lx_r\notin A(T)$ with $l\in [\alpha, r-2]$, then $x_1x_{l+1}\in A(T)$. 
 Since   $A(x_{\alpha}\rightarrow \{x_{\alpha -1}\}\cup A)=\emptyset$, there exists $j\in [\alpha +2,r]$ such that
 $x_{\alpha}x_j\in A(T)$. If $x_{\alpha}x_r\notin A(T)$, then $j\leq r-1$, $x_1x_{\alpha +1}\in A(T)$ (by the above observation) and $Q=x_0x_1x_{\alpha +1}\ldots x_{j-1}zx_2\ldots x_{\alpha}x_j\ldots x_r$, a contradiction. We may therefore assume that $x_{\alpha}x_r\in A(T)$. Since
  $d^-(x_r, \{x_{\alpha},x_{\alpha +1},\ldots , x_{r-2}\})$ $=n-k-2$ and  $|\{x_{\alpha},x_{\alpha +1},\ldots , x_{r-4}\}|=n-3\geq n-k-2$, it follows that there is an integer $i\in [\alpha, r-4]$ such that $x_ix_r\in A(T)$ and 
$x_{i+1}x_r\notin A(T)$. By our observation,  $x_1x_{i+2}\in A(T)$. Therefore, if  $x_{i+1}x_{r-1}\in A(T)$, then $Q=x_0x_1x_{i +2}\ldots x_{r-2}zx_2\ldots x_{i+1}x_{r-1} x_r$, and if  $x_{r-1}x_{i+1}\in A(T)$, then $Q=x_0x_1x_{i +2}\ldots x_{r-1}x_{i+1}zx_2\ldots x_{i} x_r$. In both cases we have a contradiction. Claim 5 is proved. \fbox \\\\
 
 From now on, by $T^1$ we denote the subtournament $T\langle\{x_1,x_2,\ldots , x_{\alpha -1}\}\rangle$.
 
 \textbf{Claim 6:} $B\not= \emptyset$.

\textbf{Proof:} By contradiction, suppose that $B= \emptyset$. This means that $A(\{x_1,x_2,\ldots , x_{\alpha -1}\}$ $\rightarrow \{x_{\alpha +2},x_{\alpha +3}, \ldots ,x_r\})=\emptyset$, i.e.,
 for any $i\in [1,\alpha -1]$, $d^-(x_i, 
\{x_{\alpha +2}, x_{\alpha +3}, \ldots , x_r\})= 2n-k-|A|-\alpha -1$.

\textbf{Case 6.1.} The subtournament $T^1$ is not  regular.  Then by Lemma 3.1(iii), there is a vertex $x\in V(T^1)$ such that $d^-(x,V(T^1))\geq 0.5(\alpha -1)$. Therefore, since $r=2n-k-|A|$ and $\alpha \leq n-\lfloor 0.5|A|\rfloor-1$, we have

 $n=d^-(x)\geq |A|+|\{x_{\alpha +2}, x_{\alpha +3}, \ldots , x_r\}|+0.5(\alpha -1)\geq  n+0.5(n-2k-2+\lfloor 0.5|A|\rfloor)$.

Since $2k\leq n-2$ and $|A|\geq 1$, we have that all  inequalities in the last inequality in fact are equalities, i.e., $|A|=1$, $2k=n-2$ ($n$ is even) and $\alpha =n-0.5\lfloor|A|\rfloor-1=n-1$. If $x_ix_r\in A(T)$ with $i\in [\alpha, r-2]$, then, since $\alpha =n-1$ and $|A|=1$, we have $d^+(x_{i+1})\geq |A|+|\{x_0,x_1,\ldots , x_{\alpha -1},x_{i+2}\}|\geq n+1$, a contradiction. We may therefore assume that $x_r\rightarrow \{x_{\alpha},x_{\alpha+1}, \ldots , x_{r-2}\}$. This together with  $x_r\rightarrow A\cup \{x_{1},x_{2}, \ldots , x_{\alpha -1}\}$ and $2k= n-2$ implies that $d^+(x_r)\geq 2n-k-2\geq n+1$, a contradiction.

\textbf{Case 6.2.} The subtournament $T^1$ is regular. Then $d^-(x_1,V(T^1))=0.5(\alpha -2)$. Since $B=\emptyset$, $2k\leq n-2$ and $r=2n-k-|A|$, we have that  $n=d^-(x_1)\geq |A|+|\{x_0,x_{\alpha +2}, x_{\alpha +3}, \ldots , x_r\}|+0.5(\alpha -2)\geq  n+0.5(n-\alpha)$, a contradiction since $\alpha \leq n-1$. Claim 6 is proved. \fbox \\\\

By Claim 4,  $\alpha \geq 2$. If $\alpha = 2$, then $B=\emptyset$ or $B=\{x_1\}=\{x_1,\ldots , x_{\alpha -1}\}$, violating Claim 5 or Claim 6. Therefore, $\alpha \geq 3$ and $A\rightarrow \{x_0,x_1,x_2,x_3\}$.  By the digraph duality, $\{x_{r-3},x_{r-2},x_{r-1},x_{r}\}\rightarrow A$. \\

\textbf{Claim 7:} $|B|\geq 2n-2k-\alpha -3$.

\textbf{Proof:} By Claim 5, $\{x_1,x_2,\ldots , x_{\alpha -1}\}\setminus B\not= \emptyset$.

Assume first that the subtournament $T^2:=T\langle\{x_1,x_2,\ldots , x_{\alpha -1}\}\setminus B\rangle$ is not  regular.  By Lemma \ref{le1}(iii), $V(T^2)$ contains a vertex $x$ such that $d^-(x,V(T^2))\geq 0.5|V(T^2)|$. On the other hand, from the definition of $B$ it follows that $A\cup \{x_{\alpha +2}, x_{\alpha +3},\ldots , x_r\}\rightarrow x$. Therefore,
$$
n=d^-(x)\geq |A|+r-\alpha -1 +0.5(\alpha -1-|B|)= n+0.5(2n-2k-\alpha -3 -|B|).
$$
Hence,  $|B|\geq 2n-2k-\alpha -3$.
  
  Assume next that $T^2$ is  regular. Then for all $x_i\in V(T^2)$,  $d^-(x_i,V(T^2))=0.5(\alpha -|B|-2)$. Let $x_q\in V(T^2)$ and $q$ is the minimum with this property. Then
$$
n=d^-(x_q)\geq |A\cup \{x_{q-1},x_{\alpha +2}, x_{\alpha +3},\ldots , x_r\}|+0.5(\alpha -|B|-2)$$ $$= n+0.5(2n-2k-\alpha -2 -|B|).
$$
Hence,  $|B|\geq 2n-2k-\alpha -2$. Claim 7 is proved. \fbox \\\\

Let  $M:=\{x_j\, | \, j\in [\alpha +2,r-1] \, \hbox{and} \, A(\{x_0,x_1,\ldots , x_{\alpha -1}\}\rightarrow x_j)\not= \emptyset\}$.

\textbf{Proposition 2:} If $x_j\in M$, then $x_{j-1}x_r\notin A(T)$.

\textbf{Proof:} If $x_j\in M$ and $x_{j-1}x_r\in A(T)$,  by the definition of $M$ there is a vertex $x_i$ with $i\in [0,\alpha -1]$ such that $x_ix_j\in A(T)$. Therefore,  $Q=x_0x_1\ldots x_ix_j\ldots x_{r-1}zx_{i+1}\ldots x_{j-1}x_r$, a contradiction. \fbox \\

\textbf{Claim 8:} $M\not= \{x_{\alpha +2},x_{\alpha +3},\ldots , x_{r-1}\}$.

\textbf{Proof:} Suppose, on the contrary, that
 $M = \{x_{\alpha +2},x_{\alpha +3},\ldots , x_{r-1}\}$. From $d^-(x_r,$ $\{x_1,x_2,\ldots , x_{\alpha -1}\})=0$ (by the condition of Case I) and Proposition 2 it follows that
$$
n=d^+(x_r)\geq |A|+|\{x_1,x_2,\ldots , x_{\alpha -1}, x_{\alpha +1},x_{\alpha +2},\ldots , x_{r-2}\}|=2n-k-3.
$$
Hence, $n-k\leq 3$. Since $2k\leq n-2$, we obtain that $n\leq 4$, which contradicts that $n\geq 5$. Claim 8 is proved. \fbox \\

From Claim 8 it follows that $\{x_{\alpha +2},x_{\alpha +3},\ldots , x_{r-1}\}\setminus M \not= \emptyset$. 

\textbf{Claim 9:} $2|M|\geq 2n-2k-\alpha -3$.

\textbf{Proof:} Assume first that the subtournament $T^1$ is not regular. Then $V(T^1)$ contains a vertex $x$ such that $d^-(x,V(T^1))\geq 0.5(\alpha -1)$. This together with $A\cup (\{x_{\alpha +2},x_{\alpha +3},\ldots , x_{r}\}\setminus M)\rightarrow x$ and $r=2n-k-|A|$ implies that 
$$
n=d^-(x)\geq |A|+0.5(\alpha -1)+r-\alpha -1-|M|=n+0.5(2n-2k-2|M|-\alpha -3).
$$
Hence, $2|M|\geq 2n-2k-\alpha -3$. 
Assume next that $T^1$ is regular. Then, $d^-(x_1,V(T^1))=0.5(\alpha -2)$. Now by a similar argument as above, we obtain
$$
n=d^-(x_1)\geq |A|+ |\{x_0\}|+0.5(\alpha -2)+r-\alpha -1-|M|=n+0.5(2n-2k-2|M|-\alpha -2).
$$
Hence, $2|M|\geq 2n-2k-\alpha -2$. Claim 9 is proved. \fbox \\\\

From Claim 9, $2k\leq n-2$ and 
$\alpha \leq n-
\lfloor 0.5|A|\rfloor -1$ it follows that  $2|M|\geq 2n-2k-\alpha -3\geq \lfloor 0.5|A|\rfloor$. Hence, if $|M|=0$ (i.e., $M=\emptyset$), then $2k=n-2$, $|A|=1$ and $\alpha = n-
\lfloor 0.5|A|\rfloor -1=n-1$. In particular, $|M|=0$ means that  $x_{r-1}\rightarrow \{x_0,x_1, \ldots , x_{\alpha -1}\}$. This together with $x_{r-1}\rightarrow \{x_r\}\cup A$  implies that  $d^+(x_{r-1})\geq n+1$, which is a contradiction. Therefore, $M\not=\emptyset$.\\

Let $W:=\{x_i \,\, \hbox{with} \,\, i\in [\alpha +1,r-1]\, | \, \hbox{there exists} \,\, j\in [\alpha, i-1]\,\, \hbox{such that} \, x_jx_r\in A(T) $ $ \hbox{and} \,\, |\{x_{j+1},\ldots , x_i\}|\leq \lfloor 0.5|\{x_{j+1}, \ldots , x_r\}|\rfloor \}$.

 By Lemma \ref{le4}, $A(\{x_0,x_1,\ldots ,x_{\alpha -1}\}\rightarrow W)=\emptyset$, i.e., $W\rightarrow \{x_0,x_1,\ldots , x_{\alpha-1}\}$.
 
  Let $x_{\alpha +s}x_r\in A(T)$, $s\geq 0$ and $s$ is the minimum with this property.\\

\textbf{Claim 10:} $|W|\geq n-k-2+\lfloor 0.5(n-|A|-\alpha -s+1)\rfloor$.

\textbf{Proof:} 
 Under the condition of Case I, we have that $x_r\rightarrow \{x_1,x_2,\ldots , x_{\alpha -1}\}$ and $d^-(x_r,\{x_{\alpha +s},x_{\alpha +s+1},\ldots , x_{r-2}\})\geq n-k-2\geq 2$.
Without loss of generality, we may assume that 
$\{x_{\alpha+s+1},x_{\alpha+s+2},\ldots , x_{\alpha+s+m},  x_{\alpha+s+m+1}, \ldots , x_{\alpha+s+m+t}\}\subseteq W$, where $m\geq 0$, $x_{\alpha+s+m}x_r\in A(T)$,   $x_{\alpha+s+m+1}x_r\notin A(T)$, $x_{\alpha+s+m+t+1}\notin W$ and $t=\lfloor(r-m-\alpha-s)/2\rfloor$. Let $Y:=\{x_{i+1} \, | \, i\in [\alpha+s+m+t+1,r-2],\, x_ix_r\in A(T)\}$. It is not difficult to see that $Y\subseteq W$, $m\geq n-k-3-|Y|$ and $|W|\geq |Y|+m+t$. Since $r=2n-k-|A|$, it is not hard to check that
$$
|W|\geq \lfloor  \frac{2|Y|+m+r-\alpha-s}{2}\rfloor\geq n-k-2+ \lfloor  \frac{n-|A|+|Y|-\alpha-s+1}{2}\rfloor.
$$
Therefore,  Claim 10 is true, since $|Y|\geq 0$.
 \fbox \\
 
We know  that
 $$
 \{x_r\}\cup A\cup W\rightarrow \{x_1,x_2, \ldots , x_{\alpha -1}\}. \eqno (13)
 $$
 
 \textbf{Claim 11:} $x_{\alpha}x_r\notin A(T)$, i.e., $s\geq 1$.
 
 \textbf{Proof:} By contradiction, suppose that $x_{\alpha}x_r\in A(T)$, i.e., $\alpha=0$. 
 
 \textbf{Case 11.1}. The subtournament $T^1$ is regular. Using Claim 10 (when $s=0$) and $2k\leq n-2$, we obtain 
 $
 d^-(x_1)\geq |\{x_0,x_r\}\cup A\cup W|+0.5(\alpha-2)\geq 0.5(3n-2k-2+|A|)> n
 $,
 which is a contradiction.
 
\textbf{Case 11.2}. The subtournament $T^1$ is not regular. If $d^-(x_j,V(T^1))\geq \alpha/2$ for some $x_j\in V(T^1)$, then using (13), Claim 10 and $2k\leq n-2$, we obtain 
$$
d(x_j)\geq |\{x_r\}\cup A\cup W|+\alpha/2\geq |A|+1+\alpha/2+n-k-2+0.5(n-|A|-\alpha)$$
$$
=0.5(3n-2k+|A|-2)>n,
$$
a contradiction. We may therefore assume that 
 for all 
$x_i\in V(T^1)$, $d^-(x_i, V(T^1))< \alpha/2$. Therefore by Lemma \ref{le1}(v), $T^1$ is almost regular. Then for all $x_i\in V(T^1)$, $d^-(x_i, V(T^1))\geq (\alpha-3)/2$.  Again using (13), we obtain 
$$
n= d^-(x_1)\geq |\{x_0,x_r\}\cup A\cup W|+0.5(\alpha-3)\geq n+ 0.5(n-2k-3+|A|).
 $$
Hence, $2k=n-2$, $|A|=1$ ($A=\{z\}$), $d^-(x_1,V(T^1))=(\alpha-3)/2$, $N^-(x_1)=\{x_0,x_r\}\cup A\cup W  \cup N^-(x_1,V(T^1))$. Since $\alpha \geq 3$ and $x_\alpha \notin W$, we have that $x_1x_{\alpha}\in A(T)$. 
If $x_{\alpha-1}x_j\in A(T)$ with $j\in [\alpha+2,r]$, then 
$Q=x_0x_1x_\alpha\ldots x_{j-1}zx_2\ldots x_{\alpha-1}x_j\ldots x_r$, if $x_{\alpha-1}x_{\alpha +1}\in A(T)$, then 
$Q=x_0x_1\ldots x_{\alpha-1}x_{\alpha+1}\ldots  x_{r-1}z x_{\alpha} x_r$. Thus, in both cases we have a contradiction. Therefore, $\{x_{\alpha+1},x_{\alpha+2}, \ldots , x_{r}\}\rightarrow x_{\alpha-1}$. Now, since $|A|=1$ and $2k=n-2$, we have
$$
n= d^-(x_{\alpha-1})\geq |\{z\}|+|\{x_{\alpha+1},x_{\alpha+2}, \ldots , x_{r}\}|+d^-(x_{\alpha -1},V(T^1))\geq n+(n-\alpha-1)/2.
 $$
Therefore, $\alpha=n-1$ since $\alpha \leq n-1$. Then, 
 $x_{\alpha+1}\rightarrow \{z,x_{\alpha+2}, x_0,x_1,\ldots , x_{\alpha-1}\}$ by Proposition 1(ii). This means that $d^+(x_{\alpha+1})\geq n+1$, a contradiction. Claim 11 is proved. \fbox \\\\

Let $F:=\{x_i \, | \, i\in [\alpha+1, \alpha+s] \,\,\hbox{and}\,\, d^-(x_i,\{x_1,x_2, \ldots , x_{
\alpha -1}\})\geq 1\}$.

 Then, $\{x_{\alpha +1}, \ldots , x_{\alpha +s}\} \setminus F\rightarrow \{x_1,\ldots , x_{\alpha -1}\}$.

\textbf{Claim 12:} $|F|\geq \lfloor 0.5(s+1)\rfloor$.

\textbf{Proof:} Recall that $x_{\alpha +s}x_r\in A(T)$, $s\geq 1$ and $x_r\rightarrow \{x_1,\ldots , x_{\alpha -1}\}$.  Suppose, on the contrary, that 
 $|F| < \lfloor 0.5(s+1)\rfloor$. Then $|\{x_{\alpha +1}, \ldots ,x_{\alpha +s}\}\setminus F|\geq \frac{1}{2}(s+1)$  (for this it suffices to consider when $s$ is even or not). Therefore, for every vertex  $x_i\in \{x_1,x_2, \ldots , x_{\alpha -1}\}$, $d^-(x_i,
 \{x_{\alpha +1}, \ldots ,x_{\alpha +s}\})\geq \frac{1}{2}(s+1)$. 
  Assume first that the subtournament $T^1$ is not regular. Then $V(T^1)$ contains a vertex $x$ such that $d^-(x,V(T^1))\geq 0.5(\alpha -1)$. Now using Claim 10, $2k\leq n-2$ and the fact that $\{x_r\}\cup A\cup  W\rightarrow x$, we obtain
$$ 
d^-(x)\geq |A\cup \{x_r\}\cup W|+0.5(\alpha -1)+0.5(s+1)\geq n+0.5|A|,
$$ 
a contradiction. 
Assume next that the subtournament $T^1$  is regular. Then $d^-(x_1,V(T^1))\\=0.5(\alpha -2)$. Therefore, as the above, we obtain
 $$
 d^-(x)\geq |A\cup \{x_0,x_r\}\cup W|+0.5(\alpha -2)+0.5(s+1)\geq n+0.5(|A|+1),
 $$
 a contradiction. Claim 12 is proved. \fbox \\
 
 \textbf{Claim 13:} $A(A\rightarrow x_{\alpha+s+1})\not= \emptyset$.
 
 \textbf{Proof:}  Suppose, on the contrary, that $A(A\rightarrow x_{\alpha+s+1})= \emptyset$. Then 
 $\{x_{\alpha+s+1}, \ldots , x_r\}\rightarrow A$. Since $W\subseteq \{x_{\alpha+s+1}, \ldots , x_{r-1}\}$ it follows that $W\rightarrow A$. Recall that $x_{\alpha+s}x_r\in A(T)$, $s\geq 1$ and $\alpha+s$ is the minimum with this property. By Proposition 1(i), we know that if $x_i\in B$ with $i\geq 2$, then $x_{i+1}x_1\in A(T)$  and $x_{i+1}\notin W$. By the definition of $F$,  $\{x_{\alpha+1}, \ldots , x_{\alpha+s}\}\setminus F\rightarrow x_1$. Now using Claims 7 and 10,  we obtain
 $$
 n=d^-(x_1)\geq |A|+|B|-1+|W|+|\{x_0,x_r\}|+|\{x_{\alpha+1}, \ldots , x_{\alpha+s}\}\setminus F|
 $$
 $$
 \geq |A|+2n-2k-\alpha-4+n-k-2+0.5(n-|A|-\alpha-s)+2+s-|F|
 $$
$$
= n+0.5(|A|+s-2|F|+5n-6k-3\alpha -8).
$$
Therefore,
$0\geq |A|+s-2|F|+5n-6k-3\alpha -8$.

Now we consider the set $W$. We know that $W\rightarrow \{x_0,x_1,\ldots , x_{\alpha-1}\}$ (by Lemma \ref{le4}). Note that if $x_i\in W$, then $i\geq \alpha+s+1$ and $x_{i-1}z\in A(T)$. Using this, it is not difficult to show  that if $x_j\in F$ and $x_i\in W$, then  $x_{i}x_{j-1}\in A(T)$.
 Thus we have, if $x_j\in F$, then $W\rightarrow x_{j-1}$. If $T\langle W\rangle$ is not regular, then for some $x\in W$, $d^+(x, W)\geq 0.5|W|$, if  $T\langle W\rangle$ is  regular, then for  $x_l\in W$ with the maximum index we have,
 $d^+(x_l, W\cup \{x_{l+1}\})\geq 0.5(|W|+1)$.   In both cases, for some $x_j\in W$,  $d^+(x_j, W\cup \{x_{l+1}\})\geq 0.5|W|$. Therefore there is a vertex $x_j\in W$ such that
 $$
 n=d^+(x_j)\geq |A|+|F|+0.5|W|+|\{x_0,x_1,\ldots , x_{\alpha-1}\}|
 $$
 $$
 \geq |A|+|F|+\alpha +0.5(n-k-2+0.5(n-|A|-\alpha -s))
 $$
$$
=(3|A|+3\alpha +4|F|+3n-2k-4-s)/4.
$$
Therefore,
$0\geq 3|A|+3\alpha +4|F|-n-2k-4-s$. 
Together with $0\ge |A|+s-2|F|+5n-6k-3\alpha-8$ and $n\ge 2k+2$, we obtain $0\ge 2|A|+|F|-2,$ which is a contradiction since $|A|\ge 1, |F|\ge 1$, and thus $2|A|+|F|-2\ge 1.$ Claim 13 is proved. \fbox \\\\

%Summing this and   $0\geq |A|+s-2|F|+5n-6k-3\alpha -8$, we obtain $0\geq 2|A|+|F|-2\geq 1$, which is a contradiction since $|A|\geq 1$ and $|F|\geq 1$. 

 Now we are ready to complete the discussion of Case I. 
 Let $q:=max\{i\, | \, A(A\rightarrow x_i)\not= \emptyset\}$. Let now $zx_q\in A(A\rightarrow x_q)$. Then $z\rightarrow \{x_0,x_1,\ldots ,x_q\}$ and $q\leq n-1$. Note that $3\leq \alpha \leq q\leq r-4$ and $\{x_{q+1},x_{q+2},\ldots , x_r\} \rightarrow A$. In this case (Case I) we have $A(\{x_1,x_2,\ldots , x_\alpha\}\rightarrow x_r)=\emptyset$ as $s\geq 1$. From $A(A\rightarrow x_{\alpha+s+1})
 \not= \emptyset$ it follows that $q\geq \alpha+s+1$. Observe that $|A|\geq 2$ since $\alpha+s > \alpha$, $A(A\rightarrow x_{\alpha+s})\not= \emptyset$ and $A(x_{\alpha +s}\rightarrow A)\not= \emptyset$. From $q\geq \alpha +s+1\geq \alpha +2$ and $x_{\alpha+s}x_r\in A(T)$  it follows that there is an integer $p\in [\alpha+s,q-1]$ such that $x_px_r\in A(T)$. Let $p$ be the maximum with this property. Now using the first part of Lemma \ref{le3}, we obtain 
 $
 A(\{x_0,x_1,\ldots , x_{p-2}\}\rightarrow \{x_{q+2},x_{q+3}, \ldots ,x_{r-1}\})=\emptyset, 
 $
 which in turn implies that (since $p-2\geq \alpha-1$)
 $$
 A(\{x_0,x_1,\ldots , x_{\alpha-1}\}\rightarrow \{x_{q+2},x_{q+3}, \ldots ,x_{r-1}\})=\emptyset. \eqno (14)
 $$
 Let $m:=max\{i\in [1,r-2]\, | \, x_ix_r\in A(T)\}$ and consider two cases depending on $m$.
 
 \textbf{Case I.1.} $m\not= p$. Then $m\geq q$ (since $p< q$). Using the facts that $\{x_{q+1},x_{q+2}, \ldots , x_r\}\\ \rightarrow A$ and $z\rightarrow \{x_0,x_1, \ldots , x_q\}$,  we obtain $x_{m+1}\rightarrow \{x_0,x_1,\ldots , x_{q-1}\}$ (for otherwise, $x_jx_{m+1}\in A(T)$ with $j\in [0,q-1]$ and $Q=x_0x_1\ldots x_jx_{m+1}\ldots x_{r-1}zx_{j+1}\ldots  x_mx_r$, a contradiction). Therefore, 
 $$
n=d^+(x_{m+1})\geq |A|+|\{x_0,x_1,\ldots , x_{q-1}, x_{m+2}\}|=|A|+q+1.
 $$
Hence, $q\leq n-|A|-1$, which in turn implies that
$
|\{x_{q+2},x_{q+3}, \ldots , x_{r-1}\}|=r-q-2\geq n-k-1.
$ 
 By the definition of the set $M$ and (14), we have that $M \subseteq  \{x_{\alpha+2}, x_{\alpha+3}, \ldots , x_{q+1}\}$.  Let $x_{j+1}\in M$. Then $\alpha+1\leq j\leq q$, and by Proposition 2, for all $i\in [q+2,r]$, $x_jx_i\notin A(T)$ (for otherwise, if $x_jx_i\in A(T)$ for some $i\in [q+2,r]$, then by the definition of $M$ there is a vertex $x_g$ with $g\in [0,\alpha-1]$ such that $x_gx_{j+1}\in A(T)$, and hence $Q=x_0x_1\ldots x_gx_{j+1}\ldots x_{i-1}zx_{g+1}\ldots x_jx_i\ldots x_r$, a contradiction). Therefore, $\{x_{q+2},x_{q+3}, \ldots ,\\x_{r-1}\}$ contains a vertex $x$ such that $n=d^+(x)\geq |A|+|\{x_0,x_1,\ldots , x_{\alpha-1}\}|+|M|+0.5(r-q-3)$. Now using Claim 9 and the facts that $r-q-2\geq n-k-1$ and $\alpha \geq n-k-|A|$, we obtain
$$
n=d^+(x)\geq |A|+\alpha+ 0.5(2n-2k-\alpha -3)+0.5(n-k-2)$$ $$\geq n+0.5(|A|+2n-4k-5)\geq n+0.5(|A|-1),
$$
which is a contradiction since $|A|\geq 2$.
 
 \textbf{Case I.2.} $m=p<q$.  
Recall that $s\geq 1$ and $d^-(x_r,\{x_{\alpha+s}, x_{\alpha+s+1},\ldots , x_{m-2}\})\geq n-k-4\geq 0$. By the first part of Lemma \ref{le3} (when in Lemma \ref{le3}, $s=m$ and $t=r$) we have 
 $$
 A(\{x_{0},x_{1},\ldots , x_{m-2}\}\rightarrow  \{x_{q+2},x_{q+3},\ldots ,x_{r-1}\})=\emptyset. \eqno (15)
 $$
It is easy to see that $|\{x_{0},x_{1},\ldots , x_{m-2}\}|\geq n-k+\alpha -3$ since $d^-(x_r,\{x_1,x_2,\ldots , x_{\alpha}\})=0$. We have that $|\{x_{q+1},\ldots , x_{r-1}\}|\geq 2n-k-|A|-q-1\geq n-k-|A|$ as $q\leq n-1$. Since $\alpha+s+1\leq m\leq q-1$ and $|A|\geq 2$, it follows that for some $y\in A\setminus \{z\}$, $x_{\alpha+s}y\in A(T)$. If $q=n-1$, then $yz\in A(T)$. Therefore, $Q=x_0x_1\ldots x_{\alpha+s}yzx_{\alpha+s+2}\ldots x_r$, a contradiction. From now on, assume that $q\leq n-2$. Then $|\{x_{q+2},\ldots ,x_{r-1}\}|\geq n-k-|A|$.

Assume first that the subtournament $T^3:=T\langle \{x_{q+2},\ldots , x_{r-1}\}\rangle$ is regular. Then, $d^+(x_{r-1},V(T^3))\geq 0.5(n-k-|A|-1)$. This together with $m\geq n-k+\alpha -2$, (15) and $\alpha \geq n-k-|A|$ implies that
$$
n=d^+(x_{r-1})\geq |A|+|\{x_0,x_1,\ldots , x_{m-2},x_r\}|+0.5(n-k-|A|-1)$$ 
$$=|A|+m+0.5(n-k-|A|-1)\geq |A|+n-k+\alpha-2+0.5(n-k-|A|-1)$$
$$\geq |A|+n-k-2+n-k-|A|+0.5(n-k-|A|-1) $$
$$
=2n-2k-2+0.5(n-k-|A|-1). 
$$
If $|A|\leq k$ or $2k\leq n-3$, then  the last inequality gives a contradiction. We may therefore assume that $2k=n-2$ and $|A|=k+1$. Then from $\alpha\geq 3$ and $\alpha \geq n-k-|A|$ it follows that  $\alpha\geq n-k-|A|+2$. Now it is not difficult to check that $d^+(x_{r-1})\geq n-k+\alpha -2+|A|+0.5(n-k-|A|-1)\geq n+2$, which is a contradiction.

Assume next that the subtournament $T^3$ is not regular. Then 
$V(T^3)$ contains a vertex $x$ such that $d^+(x,V(T^3))\geq 0.5(n-k-|A|)$. Hence by (15),  
$$
d^+(x)\geq n-k+\alpha-3+|A|+0.5(n-k-|A|). \eqno (16)
$$ 
 If $2k\leq 2n-3$, then using the last inequality, $|A|\leq k+1$ and $\alpha \geq n-k-|A|$,  we obtain that $d^+(x)\geq n+1$, which is a contradiction. We may therefore assume that $2k=n-2$. Observe that when $2k=n-2$, then from 
 $\alpha \geq 3$ and  $\alpha \geq n-k-|A|$ it follows that if $|A|=k$, then $\alpha \geq n-k-|A|+1$, and if $|A|=k+1$, then $\alpha \geq n-k-|A|+2$. If $|A|\leq k-1$, then using (16), $2k=n-2$ and $\alpha \geq n-k-|A|$, we obtain $d^+(x)> n$, a contradiction. We may therefore assume that $k\leq |A|\leq k+1$. Then by the above observation, $\alpha\geq n-k-|A|+1$, which together with (16) implies that  $d^+(x)> n$, which is a contradiction. 
   The discussion of Case I.2 is completed.\\

\textbf{Case II.} $A(\{x_1,x_2,\ldots ,  x_{\alpha -1}\}\rightarrow x_{r})\not= \emptyset$.  Let $x_px_r\in  A(T)$ with $p\in [1,\alpha -1]$. If $A(x_0\rightarrow
 \{x_{\alpha+2},x_{\alpha+3},\ldots ,  x_{r-1}\})= \emptyset$, then by considering the
converse tournament of $T$ we reduce the case to Case I since for some $\beta\in [\alpha+1,r-3]$, $\{x_{\beta+1},x_{\beta+2}, \ldots , x_{r}\}\rightarrow A$.
By the digraph duality, we may assume that $A(x_0\rightarrow \{x_{\alpha+2},x_{\alpha+3},\ldots ,  x_{r-1}\})\not= \emptyset$. Let $x_0x_i\in A(T)$ with $i\in [\alpha+2,r-1]$.  From (12) and the first part of Lemma \ref{le3} (when in Lemma \ref{le3}, $s=p$ and $t=r$) it follows that $p=1$. Since $r\geq 7$ (Fact 1), we have that $t-s\geq 6$. Thus, $x_1x_r\in A(T)$ and $x_0x_i\in A(T)$ with $i\in [\alpha+2,r-1]$, which contradicts the second part of Lemma \ref{le3}. This contradiction completes the proof of Theorem \ref{th4}. \end{proof}

\noindent\textbf{Remark 4:} The following example of a tournament of order $2n+1=9$ shows that Theorem \ref{th4} for $n=4$ is not true. Consider a  tournament $H$ such that 
$V(H)=\{x,y,u,v,z\}\cup A\cup B$, where $A=\{a_1,a_2\}$ and     $B=\{b_1,b_2\}$,  with the properties 
$x\rightarrow \{u,v,z\} \rightarrow y$, $B\rightarrow \{u,v\} \rightarrow A$, $y\rightarrow A\cup B \rightarrow x$,
$A\rightarrow z \rightarrow B$ and the following $xy, uv, vz, zu, a_1a_2, b_1b_2, a_1b_1$, $a_2b_2, b_2a_1, a_2b_1$
arcs are in $A(H)$.

It is easy to see that $H-\{z\}$ contains an $(x,y)$-path of length 3, but contains no $(x,y)$-path of length 4.

\vspace{0.5cm}

\noindent{\bf Acknowlegement} We are very grateful to the referees for their careful reading of the manuscript and a large number of useful suggestions.

\end{document}